\documentclass{article}
\usepackage{amsmath, amssymb, amsthm, graphicx}
\usepackage{times, eulervm}
\usepackage{hyperref}
\usepackage{amsrefs}
\title{The probability distribution as a computational resource for randomness testing}
\author{Bj\o rn Kjos-Hanssen}
\newtheorem{thm}{Theorem}[section]
\newtheorem{pro}[thm]{Proposition}
\newtheorem{cor}[thm]{Corollary}

\theoremstyle{definition}
\newtheorem{df}[thm]{Definition}

\newtheorem{que}[thm]{Question}

\newcommand{\mc}{\mathcal}

\renewcommand{\to}{\rightarrow}
\newcommand{\restrict}{\upharpoonright}

\newcommand{\nil}{\varnothing}
\newcommand{\Y}{\overline Y}

\renewcommand{\P}{\mathbb{P}}
\newcommand{\E}{\mathbb{E}}
\newcommand{\R}{\mathbb{R}}
\begin{document}
	\maketitle

	\begin{abstract}
	When testing a set of data for randomness according to a probability distribution that depends on a parameter, access to this parameter can be considered as a computational resource. We call a randomness test \emph{Hippocratic} if it is not permitted to access this resource. In these terms, we show that for Bernoulli measures $\mu_p$, $0\le p\le 1$ and the Martin-L\"of randomness model, Hippocratic randomness of a set of data is the same as ordinary randomness. The main idea of the proof is to first show that from Hippocrates-random data one can Turing compute the parameter $p$. However, we show that there is no single Hippocratic randomness test such that passing the test implies computing $p$, and in particular there is no universal Hippocratic randomness test.
	\end{abstract}

	\section{Introduction}

		The fundamental idea of statistics is that by repeated experiment we can learn the underlying distribution of the phenomenon under investigation. In this paper we partially quantify the amount of randomness required to carry out this idea. We first show that ordinary Martin-L\"of randomness with respect to the distribution is sufficient. Somewhat surprisingly, however, the picture is more complicated when we consider a weaker form of randomness where the tests are \emph{effective}, rather than merely \emph{effective relative to the distribution}. We show that such \emph{Hippocratic} randomness actually coincides with ordinary randomness in that the same outcomes are random for each notion, but the corresponding test concepts do \emph{not} coincide: while there is a universal test for ordinary ML-randomness, there is none for Hippocratic ML-randomness.  

		For concreteness we will focus on the classical Bernoulli experiment, although as the statistical tools we need are limited to Chebyshev's inequality and the strong law of large numbers, our result works also in  the general situation of repeated experiments in statistics, where an arbitrary sequence of independent and identically distributed random variables is studied. 

		When using randomness as a computational resource, the most convenient underlying probability distribution may be that of a fair coin. In many cases, fairness of the proverbial coin may be only approximate. Imagine that an available resource generates randomness with respect to a distribution for which the probability of heads is $p\ne 1/2$. It is natural to assume that $p$ is not a computable number if the coin flips are generated with  contributions from a physical process such as the flipping of an actual coin. The non-computability of $p$ matters strongly if an infinite sequence of coin flips is to be performed. In that case, the gold standard of algorithmic randomness is \emph{Martin-L\"of randomness}, which essentially guarantees that no algorithm (using arbitrary resources of time and space) can detect any regularities in the sequence. If $p$ is non-computable, it is possible that $p$ may itself be a valuable resource, and so the question arises whether a ``truly random'' sequence should look random even to an adversary equipped with the distribution as a resource. In this article we will show that the question is to some extent moot, as these types of randomness coincide. On the other hand, while there is a universal test for randomness in one case, in the other there is not.   This article can be seen as a follow-up to Martin-L\"of's paper where he introduced his notion of algorithmic randomness and proved results for Bernoulli measures \cite{ML}.

		It might seem that when testing for randomness, it is essential to have access to the distribution we are testing randomness for. On the other hand, perhaps if the results of the experiment are truly random we should be able to use them to discover the distribution for ourselves, and then once we know the distribution, test the results for randomness. However, if the original results are not really random, we may ``discover'' the wrong distribution.  We show that there are tests that can be effectively applied, such that if the results are random then the distribution can be discovered, and the results will then turn out to be random even to someone who knows the distribution. While these tests can individually be effectively applied, they cannot be effectively enumerated as a family. On the other hand, there is a single such test (due to Martin-L\"of) that will reveal whether the results are random for \emph{some} (Bernoulli) distribution, and another (introduced in this paper) that if so will reveal that distribution. 

		In other words, one can effectively determine whether randomness for some distribution obtains, and if so determine that distribution. There is no need to know the distribution ahead of time to test for randomness with respect to an unknown distribution. If we suspect that a sequence is random with respect to a measure given by the value of a parameter (in an effective family of measures), there is no need to know the value of that parameter, as we can first use Martin-L\"of's idea to test for randomness with respect to \emph{some} value of the parameter, and then use the fundamental idea of statistics to \emph{find} that parameter.   Further effective tests can be applied to compare that parameter $q$ with rational numbers near our target parameter $p$, leading to the conclusion that {if all \emph{effective} tests for randomness with respect to parameter $p$ are passed, then all tests \emph{having access to $p$ as a resource} will also be passed}. But we need the distribution to know \emph{which} effective tests to apply. Thus we show that randomness testing with respect to a target distribution $p$ can be done by two agents each having limited knowledge: agent 1 has access to the distribution $p$, and agent 2 has access to the data $X$. Agent 1 tells agent 2 which tests to apply to $X$. 

		The more specific point is that the information about the distribution $p$ required for randomness testing can be encoded in a set of effective randomness tests; and the encoding is intrinsic in the sense that the ordering of the tests does not matter, and further tests may be added: passing any collection of tests that include these is enough to guarantee randomness. From a syntactic point of view, whereas randomness with respect to $p$ is naturally a $\Sigma^0_2(p)$ class, our results show that it is actually an intersection of $\Sigma^0_{2}$ classes.

	\section{Definitions}
		If $X$ is a $2=\{0,1\}$ valued random variable such that $\P(X=1) = p$ (where $\P$ denotes probability)
		then $X$ is called a Bernoulli($p$) random variable.
		The Bernoulli measure $\mu_{p}$ on $2^{\omega}$ is defined by the stipulation that for each $n\in\omega=\{0,1,2,\ldots\}$, 
		\[
		\begin{aligned}
			\mu_p(\{X: X(n)=1\})&=p, \\
			\mu_p(\{X:X(n)=0\})&=1-p,
		\end{aligned}
		\]
		and $X(0),X(1),X(2),\ldots$ are mutually independent random variables.

		\begin{df}
			A $\mu_p$-ML-randomness test is a sequence $\{U^p_n\}_n$ that is uniformly $\Sigma^0_1(p)$ with $\mu(U^p_n)\le 2^{-n}$, where $2^{-n}$ may be replaced by any computable function that goes to zero effectively. 
	 
			A $\mu_p$-ML-randomness test is \emph{Hippocratic} if there is a $\Sigma^0_1$ class $S\subseteq 2^\omega\times\omega$ such that $S=\{(X,n):X\in U_n^p\}$. Thus, $U_n=U_n^p$ does not depend on $p$ and is uniformly $\Sigma^0_1$.  
		If $X$ passes all $\mu_p$-randomness tests then $X$ is \emph{$\mu_p$-random}. If $X$ passes all Hippocratic tests then $X$ is \emph{Hippocrates $\mu$-random}.
		\end{df}

		To explain the terminology: Andrzej Szczeklik, in ``Catharsis: on the art of medicine'' (\cite{Szczeklik}, page 18) writes
		\begin{quote}
			On his native island of Kos, only about three hundred kilometers from Delphi, Hippocrates did not consult the oracle to seek out the harbingers of fate, but looked at his patient's features instead. 
		\end{quote}
	
		In the definition of Hippocratic randomness, like the ancient medic Hippocrates we are not consulting the oracle of Delphi (i.e., an oracle for the real number $p$) but rather looking for ``natural causes''. This level of randomness recently arose in the study of randomness extraction from subsets of random sets \cite{MRL}. 

		We will often write \emph{$\mu_p$-random} instead of \emph{$\mu_p$-ML-random}, as we work in the Martin-L\"of mode of randomness throughout, except when discussing a conjecture at the end of this paper.

	\section{Chebyshev's inequality}

		We develop this basic inequality from scratch here, in order to emphasize how generally it holds. For an event $A$ in a probability space, we let $\mathbf 1_A$, the indicator function of $A$, equal 1 if $A$ occurs, and $0$ otherwise. The \emph{expectation} of a discrete random variable $X$ is 
		\[
			\E(X)=\sum_x x\cdot\P(X=x).
		\]
		where the sum is over all outcomes in the sample space. Thus $\E(X)$ is the average value of $X$ over repeated experiments. It is immediate that
		\[
			\E (\mathbf 1_{A}) = \P(A).
		\]
		Next we observe that the random variable that is equal to $a$ when a nonnegative random variable $X$ satisfies $X\ge a$ and $0$ otherwise, is always dominated by $X$. That is, 
		\[
			a\cdot \mathbf 1_{\{X\ge a\}} \le X.
		\]
		Therefore, taking expectations of both sides,
		\[
			a\cdot \P{\{X\ge a\}} \le \E(X).
		\]
		In particular, for any random variable $X$ with $\E(X)=\xi\in\mathbb R$ we have 
		\[
			a^2\cdot \P{\{(X-\xi)^2\ge a^2\}} \le \E ((X-\xi)^2).
		\]
		Let $\sigma^2$ denote variance: $\sigma^{2}=\E((X-\xi)^2)$. Then
		\[
			\P{\{|X-\xi|\ge |a|\}} \le \sigma^2/a^2.
		\]
		If we let $k\in\omega$ and replace $a$ by $k\sigma$, then
		\[
			\P{\{|X-\xi|\ge k\sigma\}} \le \sigma^2/(k\sigma)^2 = 1/k^2.
		\]
		This is Chebyshev's inequality, which in words says that the probability that we exceed the mean $\xi$ by $k$ many standard deviations $\sigma$ is rather small.

	\section{Results for ordinary randomness}

		We first prove a version of the phenomenon that for samples of sufficiently fast growing size, the sample averages almost surely converge quickly to the mean.
					
		\begin{pro}\label{one}
		Consider a sequence $Y=\{Y_n\}_{n\in\omega}$ of independent Bernoulli($p$) random variables, with the sample average 
		\[
			\Y_n := \frac{1}{n}\sum_{i=0}^{n-1} Y_i.
		\]
		Let $N(b)=2^{3b-1}$ and let 
		\[
			U^p_d=\bigcup_{b\ge d}\{Y: |\Y_{N(b)}-p|\ge 2^{-b}\}.
		\]
		Then $U^p_d$ is uniformly $\Sigma^0_1(p)$, and $\mu_p(U^p_d)\le 2^{-d}$, i.e., $\{U^p_d\}_{d\in\omega}$ is a $\mu_p$-ML-test.
		\end{pro} 
		The idea of the proof is to use Chebyshev's inequality and the fact that the variance of a Bernoulli($p$) random variable is bounded (in fact, bounded by $1/4$).

		\begin{proof}
		The fact that $U^p_d$ is $\Sigma^0_1(p)$ is immediate, so we prove the bound on its $\mu_p$-measure. We have 
		\[
			\E(\Y_n) = p \text{ and }\sigma^2(\Y_n) = \sigma^2/n,
		\]
		where $\sigma^2=p(1-p)\le 1/4$ is the variance of $Y_0$ and $\sigma^2(\Y_n)$ denotes the variance of $\Y_n$. Thus $\sigma\le 1/2$, and
		\[
			\P\left\{|\Y_n-p|\ge k\cdot\sigma(\Y_n)\right\} \le 1/k^2,
		\]
		so 
		\[
			\P\left\{|\Y_n-p|\ge \frac{k}{2\sqrt{n}}\right\} \le
			\P\left\{|\Y_n-p|\ge \frac{k\cdot\sigma}{\sqrt{n}}\right\} \le 1/k^2.
		\]
		Let $b$ be defined by $2^{-(b+1)}=1/k^2$. Now, we claim that $2^{-b}\ge\frac{k}{2\sqrt{n}}$ by taking $n$ large enough as a function of $b$:
		\[
			n\ge k^2 4^{b-1}=2^{b+1}4^{b-1}=2^{3b-1}.
		\]
		Thus, if $n\ge N(b):=2^{3b-1}$, 
		\[
			\P\left\{|\Y_n-p|\ge 2^{-b}\right\} \le 2^{-(b+1)},
		\]
		so
		\[
			\P\left\{|\Y_{N(b)}-p|\ge 2^{-b}\text{ for some }b\ge d\right\} \le \sum_{b\ge d} 2^{-(b+1)}=2^{-d}.
		\]
		\end{proof}

		\subsection*{Brief review of computability-theory} More details can be found in standard textbooks such as Nies \cite{NiesBook} and Soare \cite{Soare} but we provide some information here. Uppercase Greek letters such as $\Phi$ will be used for \emph{Turing functionals}, which are partial functions $B\mapsto \Phi^{B}$ from $2^{\omega}$ to $2^{\omega}$ computed by oracle Turing machines. We write $A=\Phi^{B}$ if $A\in 2^{\omega}$ is computed from the oracle $B\in 2^{\omega}$ using the Turing functional $\Phi$. Then $A(n)=\Phi^{B}(n)$ is the $n^{\text{th}}$ bit of $A$. An oracle Turing machine may, or may not, halt with oracle $B$ on an input $n$; this is written $\Phi^{B}(n)\downarrow$, and $\Phi^{B}(n)\uparrow$, respectively. 

		The \emph{Turing jump} $A'$ is the halting set (the set of solutions to the halting problem) for oracle machines with $A$ on the oracle tape. In terms of Turing reducibility $\le_{T}$, the relativized halting problem remains undecidable in the sense that $A'\not\le_{T}A$ (in fact $A<_{T}A'$). The Turing jump is not injective, even on Turing degrees (equivalence classes of the intersection $\equiv_{T}$ of the relations $\le_{T}$ and $\ge_{T}$). Indeed, there are noncomputable sets $A$ (so $C<_{T}A$ for computable sets $C$) such that $A'\equiv_{T}C'$. It is customary to pick the computable set $\nil$ for the notation here and simply write $A'\le_{T}\nil'$. Such sets $A$ are called \emph{low}.

		The \emph{low basis theorem} of Jockusch and Soare asserts that each nonempty $\Pi^{0}_{1}$ class (meaning, subset of $2^{\omega}$ that is $\Pi^{0}_{1}$ definable in arithmetic) has a low element. There is similarly a \emph{hyperimmune-free basis theorem} which asserts that each nonempty $\Pi^{0}_{1}$ class contains an element $A$ such that each function $f\in\omega^{\omega}$ computable from $A$ is dominated by some computable function in $\omega^{\omega}$. 

		These basis theorems cannot be combined, as on the one hand each low hyperimmune-free set is computable, and on the other hand there are $\Pi^{0}_{1}$ classes without computable elements (such as the set of all completions of Peano Arithmetic, or the set complement of a component $U_{n}$ of a Martin-L\"of randomness test). These facts will be used in Theorem \ref{ss-agree}.

		We are now ready for a result which in a sense sums up the essence of statistics.

		\begin{thm}
		If $Y$ is $\mu_p$-ML-random then $Y$ Turing computes $p$. \label{essence} 
		\end{thm}

		\begin{proof}
		We may assume $p$ is not computable, else there is nothing to prove; in particular we may assume $p$ is not a dyadic rational.

		Let $\{U^p_d\}_{d\in\omega}$ be as in Proposition \ref{one}. Since $Y$ is $\mu_p$-random, $Y\not\in\cap_d U^p_d$, so fix $d$ with $Y\not\in U^p_d$. Then for all $b\ge d$, we have
		\begin{equation}\label{star}
			|\Y_{N(b)}-p|< 2^{-b}
		\end{equation}
		where $N(b)=2^{3b-1}$. 

		If the real number $p$ is represented as a member of $2^\omega$ via
		\[
			p = \sum_{n\in\omega} p_n 2^{-n-1} = .p_0 p_1 p_2\cdots
		\]
		in binary notation, then we have to define a Turing functional $\Psi_d$ such that $p_n=\Psi_d^Y(n)$.

		We pick $b\ge \max\{n+1, d\}$, such that $\Y_{N(b)}=.y_0\cdots y_n\cdots$ is not of either of the forms
		\[
			.y_0\,\cdots \,y_{n} \,1^{b-(n+1)} \cdots
		\]
		\[
			.y_0\,\cdots \,y_{n} \,0^{b-(n+1)} \cdots
		\]
		where as usual $1^k$ denotes a string of $k$ ones.
		Since $p$ is not a dyadic rational, such a $b$ exists. Then by (\ref{star}) it must be that the bits $y_0\,\cdots\,y_{n}$ are the first $n+1$ bits of $p$. In particular, $y_n=p_n$. So we let $\Psi_d^Y(n)=y_n$.
		\end{proof}

	\section{Hippocratic results}

		In the last section we made it too easy for ourselves; now we will obtain the same results assuming only Hippocratic randomness. 

		\begin{thm}\label{Hippo}
		There is a Hippocratic $\mu_p$-test such that if $Y$ passes this test then $Y$ computes an accumulation point $q$ of the sequence of sample averages 
		\[
			\{\Y_n\}_{n\in\omega}.
		\] 
		\end{thm}

		\begin{proof}
		The point is that the usual proof that each convergent sequence is Cauchy gives a $\Sigma^0_1$ class that has small $\mu_p$-measure for all $p$ simultaneously. Namely, let
		\[
			V_d:=\{Y: \exists a,b\ge d\,\,|\Y_{N(a)}-\Y_{N(b)}|\ge 2^{-a}+ 2^{-b}\},
		\]
		where $N(b)=2^{3b-1}$. Then $\{V_d\}_{d\in\omega}$ is uniformly $\Sigma^0_1$. Recall from Proposition \ref{one} that we defined
		\[
			U^{p}_d=\{Y:\exists b\ge d\,\, |\Y_{N(b)}-p|\ge 2^{-b}\}.
		\]
		If there is a $p$ such that $|\Y_{N(b)}-p|<2^{-b}$ for all $b\ge d$, then 
		\[
			|\Y_{N(a)}-\Y_{N(b)}|\le |\Y_{N(a)}-p|+|p-\Y_{N(b)}|< 2^{-a}+2^{-b}
		\]
		for all $a,b\ge d$; thus we have 
		\[
			V_d\subseteq \cap_{p} U^{p}_d
		\]
		and therefore 
		\[
			\mu_p(V_d)\le\mu_p(U^{p}_d)\le 2^{-d}
		\]
		for all $p$. Thus if $Y$ is Hippocrates $\mu_p$-random then $Y\not\in V_d$ for some $d$. We next note that for any numbers $c> b$, 
		\[
			|\Y_{N(b)} - \Y_{N(c)}| < 2^{-b}+2^{-c} < 2^{-(b-1)},
		\]
		so $\{\Y_{N(c)}\}_{c\ge d}$ will remain within $2^{-(b-1)}$ of $\Y_{N(b)}$ for all $c>b$. Thus $\{\Y_{N(n)}\}_{n\ge d}$ is a Cauchy sequence and $q:=\lim_n \Y_{N(n)}$ exists. Write $q=.q_0q_1q_2\cdots$. Then
		\[
		\begin{aligned}
			|\Y_{N(b)}-q| &<  2^{-(b-1)}, \text{ so }\\
			|\Y_{N(b+1)}-q| &<  2^{-b};
		\end{aligned}
		\]
		if we define $\Theta_d$ as $\Psi_d$ in the proof of Theorem \ref{essence} except with $N(\cdot)$ replaced by $N(\cdot+1)$, then
		\[
			q_n = \Theta_d^Y(n),
		\] 
		and so $Y$ computes $q$ using the Turing reduction $\Theta_d$. 
		\end{proof}

		To argue that the accumulation point $q$ of Theorem \ref{Hippo} is actually equal to $p$ under the weak assumption of Hippocratic randomness, we need: 
		\subsection*{An analysis of the strong law of large numbers.} 

		We say that the strong law of large numbers (SLLN) is satisfied by $X\in 2^{\omega}$ for the parameter $p\in [0,1]$ if $\lim_{n\to\infty}\overline X_{n}/n= p$. The strong law of large numbers is then the statement that the set of those $X$ that do not satisfy the SLLN for $p$ has $\mu_{p}$-measure zero. 

		Let $\{X_n\}_{n\in\omega}$ be independent and identically distributed random variables with mean 0, and let $S_n=\sum_{i=0}^n X_i$. Then $S_n^4$ will be a linear combination (with binomial coefficients as coefficients) of the terms
		 \[
			\sum_i X_i^4,					\quad
			\sum_{i<j} X^3_iX_j,				\quad
			\sum_{i<j<k} X_i^2 X_j X_k,		\quad
			\sum_{i<j<k<\ell} X_i X_j X_k X_\ell,	\quad 
			\text{and}						\quad 
			\sum_{i<j} X_i^2 X_j^2.
		\]
		Since $\E(X_i)=0$, and $\E(X_i^aX_j^b)=\E(X_i^a)\E(X_j^b)$ by independence, and each $X_i$ is identically distributed with $X_1$ and $X_2$, we get 
		\[
			\E(S_n^4) = n \,\E(X_1^4) + {n\choose 2}{4\choose 2} \E(X_1^2X_2^2)
			= n \,\E(X_1^4) + 3n(n-1) \E(X_1^2)^2.
		\]
		Since $0\le\sigma^2(X_1^2) = \E(X_1^4)-\E(X_1^2)^2$, this is (writing $K:=\E(X_1^4)$)
		\[
			\le (n+3n(n-1)) \E(X_1^4) = (3n^2-2n) K,
		\]
		so $\E(S_n^4/n^4) \le \frac{3K}{n^2}$.
		Now for any $a\in\R$,
		\[
			S_n^4/n^4 \ge a^4 \cdot \mathbf 1_{\{S_n^4/n^4\ge a^4\}}
		\]
		surely, so (as in the proof of Chebyshev's inequality)
		\[
			\E(S_n^4/n^4) \ge a^4 \cdot \E(\mathbf 1_{\{S_n^4/n^4\ge a^4\}}) = a^4\cdot \P( S_n^4/n^4\ge a^4),
		\]
		giving
		\[
			\P(\overline X_n = S_n/n \ge a) \le \frac{3K}{n^2 a^{4}}.
		\]
		We now apply this to $X_n=Y_n-\E(Y_n) = Y_n - p$ (so that $K_p:=K$ depends on $p$). Note that (writing $\overline p=1-p$)
		\[
			K_p=\E[(Y_1-p)^4] = \overline p^4\cdot p + p^4\cdot \overline p = p\overline p ({\overline p}^3+p^3) \le \frac{1}{4}\cdot 2 = \frac{1}{2},
		\]
		so $\P(\exists n\ge N\,\, |\Y_n-p|\ge a)$ is bounded by 
		\[
			\sum_{n\ge N} \frac{3K_p}{n^2 a^4} \le \frac{3}{2a^4} \sum_{n\ge N} \frac{1}{n^2}
			\le \frac{3}{2a^4} \int_{N-1}^\infty \frac{dx}{x^2} = \frac{3}{2a^4(N-1)}.
		\]
		This bound suffices to obtain our desired result:

		\begin{thm}\label{harder}
		If $Y$ is Hippocrates $\mu_p$-random then $Y$ satisfies the strong law of large numbers for $p$.
		\end{thm}
		\begin{proof}
		Let $q_1$, $q_2$ be rational numbers with $q_1<p<q_2$.
		Let
		\[
			W_N:=\{Y: \exists n\ge N\,\,\Y_n\le q_1\}\cup\{Y: \exists n\ge N\,\,\Y_n\ge q_2\}.
		\]
		Then $\{W_N\}_{N\in\omega}$ is uniformly $\Sigma^0_1$, and $\mu_p(W_N)\to 0$ effectively:
		\[
			\mu_p(W_N)\le \frac{3}{2(p-q_1)^4(N-1)}+\frac{3}{2(p-q_2)^4(N-1)}=\text{constant}\cdot\frac{1}{N-1}.
		\]

		Thus if $Y$ is Hippocrates $\mu_p$-random then $Y\not\in\cap_n W_n$, i.e., $\Y_n$ is eventually always in the interval $(q_1,q_2)$. 
		\end{proof}

		\begin{cor}\label{assume-only}
		If $Y$ is Hippocrates $\mu_p$-random then $Y$ Turing computes $p$.
		\end{cor}
		\begin{proof}
		By Theorem \ref{Hippo}, $Y$ computes the limit of a subsequence $\{\Y_{N(b)}\}_{b\in\omega}$. By Theorem \ref{harder}, this limit must be $p$.
		\end{proof}

		Note that the randomness test in the proof of Theorem \ref{harder} depends on the pair $(q_1,q_2)$, so we actually needed infinitely many tests to guarantee that $Y$ computes $p$. This is no coincidence. Let $Y\ge_T p$ abbreviate the statement that $Y$ Turing computes $p$, i.e., $p$ is Turing reducible to $Y$.
        
		\begin{thm}\label{ss-agree}
		For all $p$, if there is a Hippocratic $\mu_p$-test $\{U_n\}_{n\in\omega}$ such that 
		\[
			\{X:X\not\ge_T p\}\subseteq\cap_n U_n,
		\]
		then $p$ is computable.
		\end{thm}
		\begin{proof}
		Let $\{U_n\}_{n\in\omega}$ be such a test. By standard computability theoretic basis theorems, the complement $2^\omega\setminus U_1$ has a low member $X_1$ and a hyperimmune-free member $X_2$. By assumption $X_1\ge_T p$ and $X_2\ge_T p$, so $p$ is both low and hyperimmune-free, hence by another basic result of computability theory \cite{Soare}, $p$ is computable.
		\end{proof}

		\begin{cor}
		There is no universal Hippocratic $\mu_p$-test, unless $p$ is computable. 
		\end{cor}
		\begin{proof}
		If there is such a test then by Corollary \ref{assume-only} there is a test $\{U_n\}_{n\in\omega}$ as in the hypothesis of Theorem \ref{ss-agree}, whence $p$ is computable.
		\end{proof}

		\paragraph{Comparison with earlier work.}

		Martin-L\"of \cite{ML} states a result that in our terminology reads as follows.
		\begin{thm}\label{choose}
		There is a test that is a Hippocratic ML-test simultaneously for all $\mu_p$, the passing of which implies that the Strong Law of Large Numbers is satisfied for some $p$. 
		\end{thm}

		Levin \cite{Levin} in fact states that a real $x$ passes Martin-L\"of's test from Theorem \ref{choose} if and only if $x$ is (non-Hippocratically) $\mu_p$-ML-random for some $p\in [0,1]$. This he obtains as a corollary of the following more general result: For any $\Pi^0_1$ class $\mc C$ of measures (such as the class of Bernoulli measures), there is a uniformly $\Sigma^0_1$ test $\{U_n\}_{n\in\omega}$, with $\mu(U_n)\le 2^{-n}$ for each $n$ and each $\mu\in\mc C$, such that a real $x\not\in\cap_{n}U_{n}$ iff $x$ is $\mu$-ML-random with respect to some measure $\mu\in\mc C$.

		Thus we see by Theorem \ref{choose} that we did not need infinitely many tests in the above proof of Theorem \ref{harder}, in order to get the Strong Law of Large Numbers to be satisfied for \emph{some} $p$; only to get it to be the \emph{correct} $p$.
    
	\section{Coincidence of randomness notions}
 
		We now show that Hippocratic randomness is the same as ordinary randomness for Bernoulli measures. The main idea is that since each random sequence computes $p$, it should be possible to turn a $\Sigma^0_1(p)$ test into a $\Sigma^0_1$ test.

		\begin{df}
		Let $\{O^p_n\}_{n\in\omega}$ be a universal $\mu_p$-test for all $p$, i.e. $\mu_p(O^p_n)\le 2^{-n}$ for all $p$ and $\{(p,X,n): X\in O_n^p\}$ is $\Sigma^0_1$, and if $\{\widehat O_n\}_{n\in\omega}$ is any other such test then $\cap_n\widehat O_n \subseteq \cap_n O_n$.
		\end{df}

		The existence of such a test follows from a relativization of the usual argument that there is a universal Martin-L\"of test.

		\begin{thm}\label{coincide}
		If $Y$ is Hippocrates $\mu_p$-random then $Y$ is $\mu_p$-random.
		\end{thm}
		\begin{proof}
		Let $\mc H_p$ be the set of all Hippocrates $\mu_p$-random reals and $2^\omega\setminus \mc H_p$ its complement. As in the proof of Theorem \ref{Hippo}, let
		\[
			V_d=\{Y: \exists a,b\ge d\,\,|\Y_{N(a)}-\Y_{N(b)}|\ge 2^{-a}+ 2^{-b}\}
		\]
		where $N(b)=2^{3b-1}$. Let $\Theta_d$ denote the reduction from the proof of Theorem \ref{Hippo} under the assumption $Y\not\in V_d$ there.

		We have
		\[
			\{Y: Y\not\in V_d\}\subseteq \{Y: \Theta_d^Y \text{ is total }\} 
			\subseteq \{Y: \Theta_d^Y=p\}\cup (2^\omega\setminus \mc H_p).
		\]
		Let
		\[
			D^{(d)}_{n}:=\left\{Y:  \exists k\,\,\left(\Theta^Y_{d}\restrict k\downarrow \And Y\in O_n^{\Theta^Y_{d}\restrict k}\right)\right\}.
		\]
		Then 
		\[
			D^{(d)}_n\subseteq O_n^p \cup \{Y: \Theta^Y_d\ne p\} \subseteq O^p_n\cup V_d\cup (2^\omega\setminus\mc H_p).
		\]
		Of course, $\mu_p (2^\omega\setminus\mc H_p)=0$. So 
		\[
			\mu_p(D^{(d)}_n)\le \mu_p(O^p_n)+\mu_p(V_d)\le 2^{-n}+ 2^{-d}.
		\]
		Form the diagonal $W_n=D^{(n)}_n$; then $\mu_p(W_n)\le 2^{-(n-1)}$ which goes effectively to zero, so $\{W_n\}_{n\in\omega}$ is a Hippocratic $\mu_p$-test. 

		Suppose for contradiction that $Y$ is Hippocrates $\mu_p$-random but not $\mu_p$-random. Since $Y$ is not $\mu_p$-random, for all $n$, $Y\in O^p_n$. Since $Y$ is Hippocrates $\mu_p$-random, there is a $d$ such that $Y\not\in V_d$ and for this $d$, $\Theta_d^Y=p$; in fact for all $n\ge d$, $\Theta_n^Y=p$.  Then $Y\in\cap_{n\ge d} D^{(n)}_n$. So $Y$ is not Hippocrates $\mu_p$-random. 
		\end{proof}

		\paragraph{An open problem.}
		Randomness tests $\{U_n\}_{n\in\omega}$ can be made more effective by requiring that $\mu_p(U_n)$ is a number that is computable from $p$. This is essentially Schnorr randomness \cite{Schnorr}. It is not hard to show that Schnorr randomness is sufficient for statistics to work, i.e.\ to ensure that a random sequence $Y$ computes the parameter $p$. More radically, we could require that $\mu_p(U_n)$ be actually computable. If we then relax the other side and let $U_n$ be $\Sigma^0_1(p)$, again a random sequence can be made to compute $p$, because one can ``pad'' $U_n$ to make it have $\mu_p$-measure equal to a computable number such as $2^{-n}$. 

		\begin{que}
		If we both require $U_n$ to be (uniformly) $\Sigma^0_1$ and $\mu_p(U_n)$ to be (uniformly) computable, is it still true that a random sequence must compute $p$?      
		\end{que}

\end{document}